\documentclass[11pt,a4paper,final,oneside,openright]{amsart}

\usepackage{amsfonts}
\usepackage{amsmath}
\usepackage{amssymb}
\usepackage[utf8]{inputenc}
\usepackage{hyperref}
\usepackage{color}

\usepackage{lipsum}

\usepackage{graphicx}

\def\R{\mathbb{R}}

\def\Z{\mathbb{Z}}

\def\SU{\sf{SU}}

\def\T{\mathbb{T}}

\def\CC{\mathcal{C}}

\def\Isom{\mathsf{Isom}}

\def\Eins{\mathsf{Eins}}

\def\Sym{\sf{Sym}}

\def\L{\mathcal L}

\def\PSL{{\sf{PSL}}}
\def\GL{{\sf{GL}}}
\def\O{{\sf{O}}}
\def\SO{{\sf{SO}}}
\def\SU{{\sf{SU}}}
\def\SL{{\sf{SL}}}

\def\Conf{{\sf{Conf}} }

\def\ad{\mathsf{ad}}

\def\Ad{\mathsf{Ad}}

\def\c{{\mathfrak{c}}}

\def\p{{\mathfrak{p}}}
\def\k{{\mathfrak{k}}}
\def\t{{\mathfrak{t}}}

\def\g{{\mathfrak{g}}}

\def\so{{\mathfrak{so}}}

\def\h{{\mathfrak{h}}}

\def\l{{\mathfrak{l}}}

\def\r{{\mathfrak{r}}}

\newtheorem{theorem}{{Theorem}}[section]
\newtheorem{proposition}[theorem]{{Proposition}}%[section]
\newtheorem{isom.ext}[theorem]{{Trivial isometric extension}}%[section]
\newtheorem{remark}[theorem]{{Remark}}%[section]
\newenvironment{sketchproof}{{\it Sketch of
		proof.}}{\hfill$\diamondsuit$\medskip}

\definecolor{purple}{rgb}{0.65,0.12,0.94}
\definecolor{forestgreen}{rgb}{0.4,0.64,0.13}
\hypersetup{
	colorlinks=true,
	linkcolor=red,
	filecolor=magenta,      
	urlcolor=black,
}

\begin{document}
\title{Semi-Riemannian metrics on compact  simple Lie Groups}
\author{}

 \author[A. Zeghib]{Abdelghani Zeghib }
 \address{UMPA, CNRS, 
 	\'Ecole Normale Sup\'erieure de Lyon, France}
 \email{abdelghani.zeghib@ens-lyon.fr 
 	\hfill\break\indent
 	\url{http://www.umpa.ens-lyon.fr/~zeghib/}}

\date{\today}
\maketitle

%\maketitle

\baselineskip 0,52cm

 \begin{abstract} This is a survey on left invariant semi-Riemannian metrics on compact Lie groups.

 \end{abstract}

 \tableofcontents
 
 \section{Introduction}
 
 Let $K$ be a Lie group endowed with a semi-Riemannian metric $g$. There is in general two fundamental questions that one can ask 
 in comparing the general situation to the Riemannian one:
 
 \medskip
 
 \begin{enumerate}
 
\item   Is the geodesic flow of $K$ complete, that is every geodesic in $K$ is defined for all time  (as this is the case when $g$ is Riemannian)?
 
 \bigskip
\item Is the isometry group  $\Isom(K, g)$-acting properly on $K$? This means that $\Isom(K, g)$ preserves some auxiliary Riemannian metric, say $\bar{g}$.

 \medskip
 \noindent
 Let us consider the two additional natural following questions:
 \medskip
 
\item By definition (of being  a left invariant metric), $K$ is a subgroup of $\Isom(K, g)$, but, then, what is the full isometry group of $(K, g)$? In particular, is the  isotropy of $1 \in K$  made by automorphisms of $K$?
 
 \bigskip
 
\item When is the conformal group $\Conf(K, g)$ essential, that is its action on $K$ does not preserves  a metric in the conformal class of $g$?  Observe in fact in this case (of left invariant metrics) that 
 non-essential means exactly $\Conf(K, g) = \Isom(K, g)$. 
 
 \end{enumerate}

\subsection{}  \label{extension}  In this note, we will survey this topic, by  focusing    on the case where $K$ is  a compact Lie group.    As said above, 
 $\Isom(K, g)$ is an extension in the diffeomorphism group $\mathsf{Diff}(K)$ of $K$ (where $K$ is seen as a subgroup of 
 $\mathsf{Diff}(K)$,   acting  by  left multiplication on itself). Let us observe however that existence of such extensions of $K$, say by a non-compact group $G$, is     not a 
 surpriszing matter.  Indeed if a semi-simple 
  $G$ is a Lie group, and $K$ is  its maximal compact, then $K$ acts simply transitively on $G/B$ where $B$ is a Borel subgroup of $G$. Thus the left $K$-left-action 
  of $K$ on itself identifies with the the $K$-action on $G/B$ (as $K \subset G$). Thus the $G$-action on $G/B$ is an extension of the $K$-action.  This action preserves some geometric structure, surely of parabolic type.  All  the question now is  to see if 
  this $G$-action on $K$   can preserve a (conformal) semi-Riemannian structure (this is specially related to Item (4) above)?

 %  Since $K$ is compact, 
% the meaning of that  $\Isom(K, g)$ acts properly means that it is compact. 
 
% if  $\Isom(K, g)$ is non-compact, or if it contains properly  contains 
 
\subsection{Results}
 The recent and classical literature  are summarized in  the following results:

\subsubsection{The geodesic flow} Marsden \cite{Marsden} proved that  the geodesic flow of a compact semi-Riemannian homogeneous space $(M, g)$ is complete.  It was observed in \cite{EFSZ}, that there is a Riemannian metric $\bar{g}$ on $M$,  which once seen as a scalar function on 
the tangent bundle $T M$, is a first integral of the geodesic flow of $g$. So, not only the geodesic flow is complete, but its orbits are uniformly bounded.

% \bigskip
 %$\bullet$  
 
 \subsubsection{Left Riemannian metrics}  \label{Compact extension} Regarding Item (3),  
 T. Ochiai and T. Takahashi     proved in \cite{OT}, that if $K$ is  a compact simple Lie group,  and the metric $g$ is Riemannian, then the  identity component of the isotropy group acts by automorphisms. So, up to a finite cover, 
 $\Isom(K, g)$ is contained  in $K \times K$, acting by the left and the right on $K$. 
 
 This beautiful  proof, of topological-algebraic nature,  will be recalled in some details in \S  \ref{Riemannian case}. This result is no longer true in the general semi-simple case, see the example in  \S  \ref{non-simple case} due to 
  Ozeki  \cite{Ozeki} who  proved a generalization   of \cite{OT} which can also help  to handle  the semi-simple case.

 % \bigskip
 %$\bullet$ 
 
 \subsubsection{Non-Riemannian case}
 Regarding Item (2), again if $K$ is simple, but $g$ has any signature, it was  recently proved by 
 Z. Chen,  K. Liang  and F. Zhu, that $\Isom(K, g)$ is compact. In particular, this group preserves 
 a left invariant Riemannian metric, and hence satisfies the previous description.  
 
 The beautiful proof uses  deep results from Gromov's rigid transformation groups theory. 
 We will  show  in \S \ref{pseudo simple} that this result also follows from 
 the simpler and direct techniques by  Baues-Globke-Zeghib \cite{BGZ1} (and \cite{BGZ2}).  
 We will also partially sketch this approach. 

 Semi-Riemannian compact semi-simple non-simple groups can however have non-compact isometry groups, see 
\S \ref{Non simple} for examples.  Actually, results of
 \cite{BGZ1}
 give many details about the semi-simple case, 
and tend to show that this  construction is essentially the unique way  to get examples (of non-compact isometry groups for left invariant metrics on compact semisimple groups).

 \subsubsection{Maximally symmetric metric} The two previous results can be formulated as follows:
 
 \begin{theorem}  \label{maximally symmetric}
 Let $K$ be a compact simple Lie group. Consider $g_K$, its left invariant (in fact also right invariant) 
 metric determined by the Killing form (defined on the Lie algebra $\k$). Then, $g_K$ is maximally symmetric among left invariant metrics, that is, for any left invariant metric $g$ on $K$, 
 $\Isom^0(K, g) \subset \Isom^0(K, g_K)$.  
 
 \end{theorem}
 
%  \bigskip
 %$\bullet$ 

 \subsubsection{The conformal group} The new contribution of the present article concerns the conformal  question (Item (4)). Based on the current  project on the Lichnerowicz conformal semi-Riemannian  conjecture, in a homogeneous setting \cite{BDRZ1, BDRZ2, BDZ3}, we get that 
% in the simple case, 
%  if $\Conf(K, g)$ is non-compact, then 
% $(K, g)$ is conformally flat. More generally, for $K$ compact Lie group, 
  if $\Conf(K, g)$ is essential,  then 
 $(K, g)$ is conformally flat.  This happens rarely:

 \begin{theorem} \label{Conformal Theorem}
 Let $K$ be a compact  semisimple Lie group. Assume that $\Conf(K, g)$  is essential.
 Then, up to a finite cover, 
 $K$ is $\SU(2)$ or  $\SU(2) \times \SU(2)$
 % or $\SO(2)  \times \SU(2)$. 
 The conformal group is (up to finite cover) respectively: $\SO(1, 3)$ and  $\SO(4, 4)$.
 %, and $\SO(2, 4)$. 

 \end{theorem}
 
% \begin{remarks} ${}$

  It  would be really  interesting to see if this  result can be  proved ``algebraically'', that is without using 
   the results  on the homogeneous Lichnerowicz conjecture.
   % (as we will indeed do in \S \ref{Conformal}).

 \subsubsection{Terminology: Supergroup extensions} In light of  
 \S \ref{extension} and \S \ref{Compact extension}, it becomes natural to call 
 a supergroup  (or maybe a supergroup extension) of $K$, a group $G$ that contains $K$ such that  $K$ acts freely and transitively on some homogeneous space $G/H$.
 
 The result mentioned in  \S  
  \ref{Compact extension}
 means exactly that a compact simple Lie group has a maximal compact  supergroup  which is 
 $K \times K$ (up to finite covers).  
 
 % One step in the proof of compactness of $\Iso(K, g)$ in the pseudo-Riemannian case, will be that 
%  any supergroup of $K$ is simple. 
 
   %\subsubsection{}  
   
 %  \medskip 
   As example,   $\SL_2(\R)$ is a supergroup extension of $\SO(2)$, since $\SO(2)$ acts 
   simply transitively on the circle which is a homogenous space of $\SL(2, \R)$. 
   However, $\SL_3(\R)$ is not a supergroup of $\SO(2)$ since it can not act (non-trivially) on 
   the circle (it is know that the unique simple Lie groups acting non-trivially the circle are covers of $\PSL(2, \R)$).  
   
   %(I don't have a rigorous proof for the SL_3 claim).
   
  One general construction of supergroups goes as follows (for any $K$).  
    Let $\rho: K \to \GL(F)$ be a representation, where $F$ is a finite dimensional subspace of functions on $K$. In other words $F$ is a $K$-invariant finite dimensional subspace in the space of all smooth functions on $K$ (endowed with its usual action). Let $\T$ be a circle   in $K$,  say given by a one parameter subgroup $t \to \exp t u$. 
  Consider the $F$ action on $K$,  defined by $f. k =   k \exp f(k) u \in K$. For $f $ constant, we get the $\T$-action by the right on $K$. Combining with the left $K$-action, we get a transitive action of  the semi-direct product $K \ltimes_{\rho} F$. 
  
  A similar construction is available with $\T$ replaced by a higher dimension torus $\T^d$.

  This is an example of a supergroup which always exists. It is interesting to see when this could preserve a semi-Riemannian metric?  In fact, results of \cite{BGZ1} say essentially that all  semi-Riemannian supergroups of 
 a compact semi-simple (non necessarily simple) group, are of this type.

 %\end{remarks}
  
       \section{Riemannian case} \label{Riemannian case} 
     
     \begin{theorem}  \cite{OT} \label{OT} Let $K$ be a connected  compact simple Lie group. Let $G$ be a supergroup of $K$, that is 
     $G$ acts faithfully  on $K$, and contains a copy of $K$ whose action identifies to the left action of $K$ on itself. 
      If $G$ is  connected and compact, then   a finite cover of $G$ 
     is a subgroup of $K \times K$. 
     
     \end{theorem}
     
    The  $K\times K$-action  on $K$, by the usual rule $(k_1, k_2)x = k_1 x k_2^{-1}$, has as kernel 
    $Z \times Z$, where $Z$ is the center of $K$.   Therefore, the theorem says that a supergroup is contained in 
    $K \times K/ Z\times Z$.
     
     \bigskip
     
     \begin{sketchproof} 
         \medskip
 
     $\bullet$ Let $H$ be the stabilizer of 1.  Then $G$ is   naturally homeomorphic to the product 
     $K \times H$. Indeed, $g.1$ equals $k \in K$, and hence $k^{-1} g = h$, for some $h \in H$. This coherently defines a bijective map $g \to (k, h)$, which is naturally continuous.

          \medskip
     $\bullet$   As a compact group,  the universal cover of $G$ decomposes as a direct product of  compact simple groups
  and an abelian group covering  a toral factor. Recall  here that a compact  group has a finite fundamental group exactly if it is 
  semi-simple. Since $K$ is simple, it has no non-trivial homomorphism to an abelian group, and hence,
  its universal cover is contained in the product of simple factors. If we change $G$ accordingly, that is we remove the toral factor from it, 
  we do not change our problem, that is this new  $G$ is still a supergroup of $K$. So, we will henceforth assume that $G$ is semi-simple.

         \medskip 
  $\bullet$    We  have equality of  homotopy groups: $\pi_i(G) = \pi_i(K) \times \pi_i(H)$.
  
  We deduce first that $\pi_1(H)$ has a finite fundamental group and is thus semi-simple. 
     
     At this stage, we can, and will, assume that all groups $G, H$ and $K$ are simply connected (and hence in particular admit direct decompositions into simple factors). 
   
        \medskip
   $\bullet$   Now, we recall that, for a simply connected simple Lie group, $\pi_2 = 1$, and $\pi_3 = \Z$. 
   This was proved by R. Bott as a corollary of the main results of \cite{Bott}), and as application of Morse theory
   to the topology of Lie groups. We don't know if a direct proof is available. 
   % (page 253)

   Thus, 
    for a  simply connected compact semi-simple group, its  $\pi_3$-group is $\Z^d$, where $d$
   is the number of its simple factors.

   Therefore, if $G$ has $d$ factors, then $H$ has $(d-1)$ factors.
   %  ideals in the decomposition of its Lie algebra, or equivalently simple factors in its universal cover. 
     
  % From this we infer that $\tilde{H}$ has $d-1$ factors.

          \medskip
     $\bullet$  Write $G = G_1 \times \ldots \times G_d$.  If $K$ projects trivially on some factor $G_i$, then we can remove this factor without changing our problem (that is we will still have a supergroup of $K$). So assume $K$ projects injectively in each $G_i$, in particular 
   for any $i$,   $\dim G_i \geq \dim K$, and $G_i$ isomorphic to $K$, in case of equality of dimensions.

          \medskip
     $\bullet$ We have $\dim H = \Sigma \dim G_i - \dim K$. So, if $\dim G_i > \dim K$, then $\dim G_i + \dim H > \dim G$, which implies
     that $G_i \cap H \neq 1$ (this happens at the Lie algebra level, and then applies to groups too).

          \medskip
     $\bullet$  Any non-trivial $G_i \cap H$ will be a non-trivial normal subgroup of $H$, and is  hence product of factors of $H$.

          \medskip
     $\bullet$  Change notations and write $G = A_1  \times \ldots  \times A_a \times B_1 \times  \ldots \times  B_b$, where $\dim A_i > \dim K$, and $B_i$ is isomorphic to $K$, for any $i$.      
     
          \medskip
     $\bullet$  Each $A_i \cap H$ is a product of factors of $H$. 
     %Any other factor of $H$ 
     
        \medskip
     $\bullet$ Consider the set $\Sigma$ of factors $H_j$ of $H$  that are contained in $A_1 \times \ldots  \times A_a$. There are at least 
     $a$ elements of $\Sigma$. Their contribution in $\dim H$ is at most $\Sigma \dim A_i$. 
     
     Remember that $H$ has $d-1 = a+b - 1$ factors, so it remains at most $b-1$
   factors  $H_j$ not in $\Sigma$. Any such $H_j$  projects non-trivially on  some $B_i$, and hence is isomorphic to 
     $K$.   The total contribution  
  of such factors in $\dim H$ is  thus at most $(b-1) \dim K$. But $\dim H= \dim G - \dim K = \Sigma \dim A_i + (b-1) \dim K$. 
  It follows that $A_1  \times \ldots  \times A_a $ is contained in $H$.

        \medskip
     $\bullet$ Remember however that $H$ is the isotropy of the $G$-action on $K$, and by the faithfulness (tacit) hypothesis, 
     $H$ contains no normal subgroup of $G$. We infer from all this, that all the $G$-factors are isomorphic to 
     $K$ and none of them intersects non-trivially $H$. Say $G = A^d$, with $A$ isomorphic to $K$

        \medskip
     $\bullet$ $H$ has $(d-1)$ factors, all embed in $A$, but since $\dim H = \dim G - \dim A = (d-1) \dim A$, each of these factors is isomorphic to $A$, that is $H$ is isomorphic to $K^{d-1}$.

        \medskip
     $\bullet$ To fix ideas, assume $d = 3$. So $G = K^3$, and $H \cong K^2$ embeds in $K^3$. This consists in two copies of $K$ in $K^3$ which commutes. So their projections on each factor of $K^3$ commute. But such a projection is either $\{1\}$ or $K$. It cannot be $K$ since 
     $K$ is not commutative. This implies these two copies cannot have same non-trivial projection on a factor of $K^3$. Hence 
     at least one of these copies is a factor of $K^3$. But then, the isotropy $H$ contains a normal subgroup of $G$ which contradicts faithfulness. This argument applies in a similar way to any situation $d \geq 3$. 
     
        \medskip
     $\bullet$ It remains to consider the case $d = 2$, so $G = K \times K$,  and $K$ and $H$ embeds ``obliquely'' in $K\times K$, and so 
     each of them is the graph of a homomorphism $K \to K$. The same applies to their Lie sub-algebras. They  are graphs in $\k \oplus \k$
     of derivations $d_1, d_2: \k \to \k$.  The intersection of these graphs consists of vectors of the forms 
     $u \oplus d_1(u) \in \k \oplus \k$, such that $d_1(u) = d_2(u)$. But $d_1-d_2$ is a derivation of $\k$, and since $\k$ is semi-simple, 
     $d_1-d_2=  \ad_w$, for some $w \in \k$. In particular $d_1(w) = d_2(w)$, and the two graphs have a non-trivial intersection. 
     This contradicts that     fact that $G$ equals $KH$  which implies the sub-algebras of $K$ and $H$ are transversal. 
     
         \medskip
     $\bullet$ All this implies that $K$ is in fact a factor of $G$.  So, all the other factors  of $G$ commute with $K$, and thus their action consist  in right multiplication,  and so, the $G$-action on $K$ transits via $K \times K$ (up to a cover). 
      \end{sketchproof}

  \section{Conformal Group, Proof of Theorem \ref{Conformal Theorem}}
\label{Conformal}

  Recall that $\Eins^{p, q}$
  is the substratum of the flat conformal semi-Riemannian geometry of signature $(p, q)$. One model of it can be defined as follows. Consider the pseudo-Euclidean space $\R^{p+1, q+1}$, and $\CC^{p+1, q+1}$ its light cone (the space of isotropic vectors). Then, $\Eins^{p, q}$ is the quotient of the light cone by the radial $\R^+$-action. 
  
   One sees in particular that $\Eins^{p, q}$ the topology of $\mathbb S^p \times \mathbb S^q$.
 
  The
  orthogonal group $\O(p+1, q+1)$ of $\R^{p+1, q+1}$ acts conformally on $\Eins^{p, q}$, and in fact equals it full conformal group (for $p+q>2$). 
  
  A  semi-Riemannian conformally flat manifold of signature $(p, q)$ is modelled on  $\Eins^{p, q}$, and conversely. In other words being conformally flat is equivalent of  having a $(G, X)$-structure, for 
  $X = \Eins^{p, q}$, and $G = \O(p+1, q+1)$. 
  .    
  
  By the results of \cite{BDRZ1, BDRZ2, BDZ3}, if $\Conf(K, g)$ is essential, then $(K, g)$ is conformally flat. 
  So, we have a developing map $\tilde{K} \to \widetilde{\Eins^{p, q}}$, where $(p, q)$, $q \geq p$,  is the signature of $g$.
 % , and 
  %$\Eins^{p, q}$, is the Einstein semi-Riemannian universe of signature $(p, q)$. 
  
 Since semisimple,  $K$ has a finite fundamental group, so up to a cover, we can assume $K$ is simply connected.

 The developing map is a local diffeomorphism, and $K$ is compact and simply connected, if follows that it is a covering, and that $K$ is the universal cover of $\Eins^{p, q}$. This implies that $p \neq 1$, since $\Eins^{1, q}$ has 
  a non-compact universal cover, and that $d$ is a diffeomorphism, since $\Eins^{p, q} $ is simply connected for $p \neq 1$. 
  
  So, from the topological viewpoint, $K$ is a semi-simple Lie group diffeomorphic to $\mathbb S^p \times \mathbb S^q$ $p \leq q$). 
  
  As recalled in Section  \ref{Riemannian case}, a semi-simple Lie group $K$ satisfies $  \pi_2(K) = 1$, and $\pi_3(K) \neq 0$ \cite{Bott}. 
  In fact, $\pi_3(K) = \Z^k$, where $k$ is the number of simple factors of $K$.

    It follows that 
  either  $(p, q)$ is either $(0, 3)$   or $p = 3$ and $q \geq 3$.
   Let us consider the case $p = 3, q \geq 3$, the other case   being   easier to handle.  
Therefore,  either $q > 3$, and then $K$ is simple, or $q = 3$, and $K$ has two simple factors. 
  
  $(K, g)$ is conformally isomorphic to $\Eins^{p, q}$, so its conformal group embeds in $\O(p+1, q+1)$. 
  In particular, the $K$-left action on itself gives a transitive action on $\Eins^{p,q}$, say via an embedding $h: K \to 
  \O(p+1, q+1)$.  Up to conjugacy,  $h$ has values in the maximal compact subgroup 
  $\O(p) \times \O(q)$. Write $h = (h_1, h_2)$.  
     
%Observe that  $h(K)$, the image of $K$ by the holonomy homomorphism is a compact subgroup of 
  %$\SO(3+1, q+1)$, and hence up to conjugacy, can be assumed inside the maximal group 
  %$\SO(4) \times \SO(q+1)$, that is $K$ acts simply transitively on $\mathbb S^p \times \mathbb S^q$, via  $h = (h_1, h_2)$.  
  
  Assume $q >3$, so $\pi_3(K) = \Z$  (since diffeomorphic to $\mathbb S^p \times \mathbb S^q$). Hence $K$ is simple. Necessarily, one of the homomorphisms $h_1$ or $h_2$ is trivial. But then 
  $h(K)$ does not act transitively on $\mathbb S^3 \times \mathbb S^q$.  
  
  From this, we infer that $q= 3$, and 
  $K$ has two simple factors. In this case $h = (h_1, h_2)$ maps $K$ to $\SO(4) \times \SO(4)$. If $K$ acts transitively 
  on $\mathbb S^3 \times \mathbb S^3$, then none of the $h_i$ is trivial.   Recall here  that up to finite cover 
  $\SO(4) = \SO(3) \times \SO(3)$. As $h(K)$ has exactly two simple factors,  each of  factors must be, up to finite covers, 
  $\SO(3)$. One can also see that, up to finite covers, $h(K)$ is a  product  of two copies of $\SO(3)$, each contained
  in one factor $\SO(4)$ (of the maximal compact). $\Box$

%  then necessarily $h(K)$ has a factor in each $\SO(4)$ 

        \section{Non Riemannian   case}  \label{pseudo simple}
        
     Compact simply connected homogeneous semi-Riemannian manifolds were studied in 
     \cite{BGZ1} (and also the unpublished \cite{BGZ2} which becomes then part of \cite{BGZ1}).  The principal result is stated as follows:

\begin{theorem} \label{BGZ-Theorem}
Let $M$ be a connected and simply connected pseudo-Riemannian
homogeneous space of finite volume, $G=\Isom(M)^\circ$,
and let $H$ be the stabilizer subgroup in $G$ of a point in $M$. 
%The semi-simple Levi factor is compact. 
Let $G= C R$ be a Levi decomposition, where $R$ is the solvable radical
of $G$.
Then:
\begin{enumerate}
\item
$M$ is compact.
\item
$C$ is compact and acts transitively on $M$.
\item
$R$ is abelian.
Let $A$ be the maximal compact subgroup of $R$. Then $A=\mathrm{Z}(G)^\circ$, the identity 
component of the center of $G$.
More explicitly, $R=A\times V$ where $V\cong\R^n$ and $V^{C}= 0$ (that is the $C$-representation has no 
 factor where it acts trivially). 
\item
$H$ is connected. 
If $\dim R>0$, then $H=(H\cap C) E$, where $E$ and $H\cap C$ are
normal subgroups in $H$, $(H\cap C)\cap E$ is finite,
and $E$ is the graph of a non-trivial homomorphism
$\varphi:R\to K$, where the restriction $\varphi|_A$ is injective.
\end{enumerate}
\end{theorem}
 
      \subsection{Sketch} Let us give some hints on the proof of this result, especially the fact that 
      $G$ has no non-compact semisimple factor (Item (2) in the Theorem).  So $G$ acts on $M$ transitively. For $X $
      in the Lie algebra $\g$, let $\bar{X}$ be the associated vector field on $M$. For $x \in M$, we define a degenerate metric 
      $m(x)$ on $\g$, using the pull back by  evaluation  map $X \in \g \to  \bar{X}(x) \in T_xM$, that is 
      $m(x)(X, Y) = g_x(\bar{X}(x), \bar{Y}(x))$ ($g$ is the given  semi-Riemannian metric on $M$). Let also $\h(x)$  be the Lie algebra of the stabilizer 
      of $x$. Observe that the Kernel of $m(x)$ is exactly $\h(x)$. Like this, one define maps $m: M \to \Sym(\g)$
      and $l: M \to \L(\g)$, where $\Sym(\g)$ is the space of quadratic forms on $\g$
      and $\L (\g)$ is the Grassmann of linear $d$-subspaces of $\g$, where $d = \dim \h(x)$ (for any $x$).
      
      The point is that $m$ and $l$ are equivariant with respect to the action of $G$ on $M$, and its natural action of 
      $\Sym(\g)$ and $\L(\g)$.

      Let us assume here that $M$ is compact (instead of the slightly more general hypothesis $M$ of finite volume as in the theorem). 

The image $Z = m(M) \subset \Sym(\g)$ is in particular  invariant under the linear $G$-action. So we are in a situation of a compact set $Z $ in a linear space, say $\R^N$, invariant by a subgroup $ G \subset \GL(N, \R)$. Let $p(t) = e^{t A}$ be a one parameter group in $G$. 
Then, for any $z \in Z$, $p(t) z$ is bounded when $t \to \infty$. Assume $A$ nilpotent, then 
$p(t) = 1 + t A + (t^2/2) A^2+ \ldots + (t/ N!) A^N$. Clearly, $p(t)z $ bounded, implies $A(z) = 0$, that is $p(t) z = z, \forall t$. If $S $ a subgroup of $G$ is generated by such one parameter groups, with $A$ nilpotent, then $S$ acts trivially on $Z$. This applies in particular to 
the semi-simple factor of $G$ of non-compact type, as well as  to the nilradical of $G$.

The case of the map $l$ is more complicated since it  has values in a Grassmann space $V$, which is compact. In this case, one 
uses another dynamical idea. Again assuming $A$ nilpotent, then a point $v$ is recurrent if there is  a sequence $t_n \to \infty$, such that 
$p(t_n) v \to v$.  One concludes in this case that $p(t)v = v$. Since the $G$-action on $M$ preserves the semi-Riemannian measure, there is a  $G$-invariant measure with full support in the image of $l$. 
By Poincar\'e recurrence Theorem, almost all points are recurrent. 
Therefore, we have the same conclusion that $S$ acts trivially 
on the image of $l$, once it is generated by one parameter groups with nilpotent infinitesimal generator. 

Let $S$ be a semi-simple factor of   non-compact type, and $\mathfrak{s}$ its Lie algebra. The  last conclusion translates in terms of brackets
  to $[\mathfrak{s}, \h(x) ]
  \subset \h(x)$
(for any $x$).  By considering the projection
  $\pi: \g \to \mathfrak{s} $, one sees that if 
the projection of $\h(x)$ is non-trivial, then this projection is an  ideal  of $\mathfrak{s}$. In particular, since there are only finitely many ideals of $\mathfrak{s}$,  we get a factor  $\l$ contained in the projection of all $\h(x)$, $\forall x$.  By semi-simplicity, this gives $\l^\prime \cong \l$, a subalgebra of   $\g$ contained in all the $\h(x)$, contradicting the faithfulness of the $G$-action. Therefore $\h(x) $ is contained in 
$\mathfrak{c} + \r$, where $\mathfrak{c}$ is a compact semi-simple factor and $\r$ is the radical. 

At the Lie group level, let $G = S. C. R$, then, for the isotropy $H \subset C. R$. So, we have a well defined map 
$M =  G/H=  (S.  C. R)/H \to S$. By compactness of $M$, $S $ must be trivial. $\Box$

 \subsection{Simple case} 
 Recall the  result of \cite{ZCL}: 
 
 \begin{theorem} \cite{ZCL} A left invariant semi-Riemannian metric on a compact simple group, has a compact isometry group.

 \end{theorem}
 
 We will deduce this result from Theorem \ref{BGZ-Theorem}, without using neither 
 \cite{Gromov, DG} nor 
 \cite{OT}. 
 
 \begin{proposition} \label{Proposition pseudo}
 If $\dim R >0$, then there is no simple subgroup $K \subset G$ which acts transitively on $M$.
 
 \end{proposition}
 
 \begin{proof}

 Let $C_0$ be the kernel of the representation of $C$ in $V$. This is a normal subgroup of $C$
 and we have a splitting $C = C_0 C_1$.  From Proposition 9. 6 (see also the proof of   Lemma 10.3) of 
 \cite{BGZ1}, we have that $H \subset C_0R$.  In particular if $C_0 = 1$, then $H \subset R$, but this is impossible since 
 $C$ acts transitively on $M$, unless $H = R$, which is also impossible since the $G$ action is (tacitly!) assumed
 to be faithful. 
 
 Assume now that $G$ contains a simple Lie group $K$ acting transitively on $M$. Then up to conjugacy, 
$K$ is a subgroup of $C$, and by simplicity, it is either in $C_0$ or in $C_1$.

Assume $K \subset C_1$. Consider $G_1 = K \ltimes R$. The isotropy $H_1 = G_1 \cap H$ is contained in $R$, since $H \subset C_0 R$. But $K$ acts transitively on $M$, so the isotropy $H_1$ must be equal to $R$, which contradicts faithfulness.

Therefore $K \subset C_0$. Consider the direct product $G_2 = K \times R$. So, on $M$, the $R$-action commutes with the transitive action of $K$.  \\
If the $K$-action on $M$ is free, in which case $M$ is identified to $K$ acting 
by left translation on itself, then $R$ must act on the right 
 via a homomorphism in $K$. So the $K \times R$-action extends to a $K \times K$-action, and thus the isometry group is compact. Finally, the case where the $K$-action on $M$ is not free,  works similarly, with a slightly more complicated notations. 
  \end{proof}

   \section{Non-simple examples}  \label{Non simple}
 
       \subsection{Riemannian Non-simple  example}  \cite{Ozeki}
   \label{non-simple case} 
 Let $L$ be any group. Embed it       
         in $L^3$ as  $A = \{(x, x, 1) / x \in L\}$, and embed $L^2$  as 
       $B = \{ (x, y, x)/ l_1, l_2 \in L\}$.  Any element of $L^3$ can be uniquely decomposed as a product of an element of 
       $A$ and an element of $B$. So $L^3$ is acts on $L^3/A$, which is identified to $B$, and on $L^3/B$ which is identified to $A$. 
       Observe however that the $L^3$-action on $A$ is not faithful. The $L^3$-action on $B \cong L^2$ is however faithful and this supergroup of $B = L^2$ is not contained (up to covers) in $B \times B$.

\subsection{Non-Riemannian non-simple example} \cite{BGZ1}
\label{ex:not_easy_group}
Let 
$G_1=(\widetilde{\SO}(3)\ltimes\R^3)\times\T^3$,
where $\widetilde{\SO}(3)$ acts on $\R^3$ by the adjoint  representation ($\R^3 \cong \so(3)$)
and let $\Phi:\R^3\to\T^3$
be a homomorphism with discrete kernel. Note $V_1$ (resp. $V_0$) the Lie algebra of $\R^3$ (resp. $\T^3$). 

Put
$H = \{(1,v,\Phi(v))\mid v\in\R^3\} 
$. Its Lie sub-algebra is $\h= \{ (0, v, \varphi(v)), v \in V_1\}$, where $\varphi: V_1 \to V_0$ is the derivative of $\Phi$.
Define  a pseudo-product  $ \langle  ,  \rangle $ on $\g_1$ by: 

- $\so(3)$ and $V_0 \oplus V_1$ are (totally) isotropic.

- if $u \in \so(3), v_0 \in V_0, v_1 \in v_1$, then $\langle u, v_0 + v_1 \rangle =  \kappa(u, v_1 + \varphi^{-1}(v_0))  $, where $\kappa$ is the Killing form of $\so(3)$, and $V_0$  and $V_1$ are  identified to $\so(3)$. 

One can check that the kernel of this product is exactly $\h$, and that the so defined product on $\g_1/\h$ has signature $(3, 3)$.
 Also, this product is $\Ad (H)$-invariant. All these properties are a particular case of the following general construction. Let $L$ any group,
 with $\l$ its Lie algebra and $\l^*$ its dual. Consider the semi-direct product $ P= L \ltimes \l^*$. Its Lie algebra as a vector space 
 is $\p =  \l \oplus \l^*$. The paring of $\l$ and its dual $\l^*$, that is $k(x, \alpha) = \alpha(x), x \in \l, \alpha \in \l^*$, determines a pseudo-scalar product 
 on $\p$ of signature $(d, d)$, $d = \dim L$, which is in fact $\Ad(P)$-invariant.

%The Lie algebras $\g_1$ of $G_1$ and $\h$ of $H$ are the
%corresponding Lie algebras from Example \ref{ex:not_easy}.

%We can extend the nil-invariant scalar product $\met$ on $\g_1$
%from Example \ref{ex:not_easy} to a left-invariant tensor on $G_1$,

From this scalar product on $\g_1/\h$, we get a $G_1$-invariant semi-Riemannian metric 
of signature $(3, 3)$ on $M_1= G_1/H$. This $M_1$ is identified to $\widetilde{\SO}(3) \times \T^3$.

and thus obtain a $G_1$-invariant pseudo-Riemannian metric of signature
$(3,3)$ on the quotient $M_1=G_1/H=\widetilde{\SO}(3)\times\T^3$.
Here, $M_1$ is a non-simply connected manifold with a non-compact
connected stabilizer.

In order to obtain a simply connected example, embed $\T^3$ in a simply
connected compact semisimple group $C_0$, for example
$C_0=\widetilde{\SO}(6)$, so that $G_1$ is embedded in
$G=(\widetilde{\SO}(3) \ltimes \R^3) \times C_0$.

%As in Example \ref{ex:not_easy2}, we can extend $\met$ from $\g_1$
The previously defined scalar product on $\g_1$ can be extended to $\g$ as follows.
Choose $\t^\prime \subset \c_0$, as a $\T^3$-invariant  supplementary subspace of the Lie algebra of $\T^3$ in that of $C_0$, and endow it with 
a positive scalar product.  Then, equip  $\g = \g_1 \oplus \t^\prime$, with the direct sum of scalar products. This is
$\Ad(H)$-invariant. We therefore get $M=G/H \cong \widetilde{\SO}(3)\times C_0$. 
Therefore $\widetilde{\SO}(3) \times \widetilde{\SO}(6)$ admits a left invariant semi-Riemannian metric  having a  non-compact 
isometry group.

\section{More results and questions}

Let us end with the following questions, some of which are good exercises.

\subsubsection{Finite isometry groups} For a compact homogeneous semi-Riemannian manifold, the isometry group 
has finitely many connected components. Observe that all results here concern the identity component. For instance, 
for a simple group $K$,  with a left invariant metric $g$, a priori, it might happen that the isotropy at the identity contains a finite group 
acting by isometry that are not automorphisms?

\subsubsection{Non-simple case} 
 There is in fact in \cite{BGZ1} more details  about isometry groups of compact simply connected
 semi-Riemannian spaces, 
 which might allow one to  an optimal classification
 of compact simply connected homogeneous semi-Riemannian manifolds, 
 in particular in the case where $M$ is identified to a  compact semi-simple Lie group. Also
 Ozeki's \cite{Ozeki} and Koszul's \cite{Koszul} results might be helpful in this regard.

    \subsubsection{Non-group case} Our proof  of Proposition \ref{Proposition pseudo} applies also to semi-Riemannian homogeneous
    spaces of simple groups, that is,  manifolds $M = K/P$, where $K$ is a compact simple group. Their isometry group is compact.

    \subsubsection{Non semi-simple case}  The conformal Theorem \ref{Conformal Theorem} generalizes,  up to a slight  modification, to the case where 
    $K$ is compact but not necessarily semi-simple. So $K$ is, up to a finite cover, the product of a semi-simple by a torus.
    As example, we have  $\SO(2)  \times \SU(2)$    whose conformal group 
is  $\SO(2, 4)$.  

    \subsubsection{The non-compact Riemannian case} For  a semi-simple group $S$  endowed with 
    a left invariant Riemannian metric $g$, $S$ is co-compact in $\Isom(S, g)$. But if $S$ contains 
    no compact factor, then it is cocompact only in groups of the form $S \times L$, with $L$ compact. In particular
    $S$ cannot be co-compact in another different semi-simple group without compact factors (see \cite{Gordon} for 
    proofs)
    
    %$\Isom(S, g)$  is a product 
   
    \subsubsection{The non-compact semi-Riemannian case}  If $S$ is simple non-compact, a semi-Riemannian left invariant metric can have a large  isometry group, say where $S$ is not compact, the is the isotropy (at the identity) is not compact.
    As example, the Killing form, determines a  bi-invariant metric, the identity component of its  isometry group 
    is $S \times S$, modulo the center.  Here, one can ask if it is a maximally symmetric metric as in Theorem  \ref{maximally symmetric}.

 %\begin{enumerate}

%\item  It is interesting to see if this  result can be  proved ``algebraically'', that is without using  the results  on the homogeneous Lichnerowicz conjecture?

 \end{document}